\documentclass[10pt]{article}
\usepackage{amsmath,amssymb}
\usepackage{color}
\usepackage{mathrsfs}
\usepackage{mathptmx}
\usepackage{verbatim}
\usepackage{epsfig}
\usepackage{CJK}
\usepackage{bm}
\usepackage{amsfonts}
\usepackage{amsthm}
\input{epsf.sty}
\usepackage{epsfig}
\usepackage{setspace}
\def\<{\langle}
\def\>{\rangle}

\def\be{\begin{equation}}
\def\ee{\end{equation}}
\def\ba{\begin{array}}
\def\ea{\end{array}}

\newtheorem{thm}{Theorem}[section]

\newtheorem{lem}[thm]{Lemma}

\newtheorem{defn}[thm]{Definition}
\newtheorem{rem}{Remark}

\numberwithin{equation}{section}

\def\be{\begin{equation}}
\def\ee{\end{equation}}
\def\br{\begin{eqnarray}}
\def\er{\end{eqnarray}}

\title{Long time stability of KAM tori for nonlinear wave equation
\footnote{Supported by  the NNSFC 11101059, NNSFC 11201292, SNSFC 12ZR1444300, ZZegd12022 and NNSFC 11301358.}}
\author{$\mbox{Hongzi \ Cong}^1$ \hspace{0.2 cm}
        $\mbox{Meina \ Gao}^2$ \hspace{0.2 cm}
        $\mbox{Jianjun \ Liu}^3$ \hspace{0.2 cm}
        \\
${}^{1}$ School of Mathematical Sciences, Dalian University of
Technology \\
${}^{2}$ School of Science, Shanghai Second Polytechnic University\\
${}^{3}$ School of Mathematical Sciences, Sichuan University}
\begin{document}
\maketitle
 {\bf Abstract}: {\it  It is proved that the KAM tori (thus quasi-periodic solutions) are long time stable for infinite dimensional Hamiltonian systems generated by nonlinear wave equation, by constructing a partial normal form of higher order around the KAM torus and showing $p$-tame property persists under KAM iterative procedure and normal form iterative procedure.}

  \maketitle
\section{Introduction and main results}

Since the initial work \cite{K3,K2,K1,W} of infinite dimensional KAM
theory by Kuksin and Wayne, there has been a lot of work about the
existence of KAM tori for the nonlinear wave equation (NLW)
\begin{equation}\label{1}
u_{tt}-u_{xx}+V(x)u+u^3+h.o.t.=0
\end{equation}
subject to Dirichlet boundary conditions $u(t,0)=u(t,\pi)=0$. For
examples, Wayne \cite{W} obtains the existence when $V(x)$ does not
belong to some set of ``bad" potentials; Kuksin \cite{K1} considers
parameterized potentials $V(x,\xi)$ and shows that there are many
quasi-periodic solutions for ``most" parameters $\xi$'s; P\"{o}schel
\cite{P3} proves that the potential $V=V(x)$ can be replaced by a
fixed constant potential $V\equiv m$; Yuan \cite{Yuan3} shows the
existence of KAM tori for any prescribed non-constant potential. All
of these results are obtained by the classic KAM iteration which
involves the so-called second Melnikov conditions. Thus, every KAM
torus is linearly stable, and around it a normal form of order 2 is
obtained. Based on the normal form, one can directly obtain
$\delta^{-1}$ long time stability of KAM tori, where $\delta$ is the
distance between initial data of the solutions and KAM tori.

A natural question is whether the KAM tori are stable in a longer
time such as $\delta^{-\mathcal{M}}$ for any given $\mathcal{M}\geq
0$. There has been a lot of work on the long time stability of the
equilibrium point $u=0$ and some approximate invariant tori for
partial differential equations. For examples, see
\cite{Bam1,Bam3,BDGS,BG,BN,B6,DS,GIP}. It is proposed by Eliasson
\cite{Elia2} that whether such stability results can be proven in a
neighborhood of a given KAM torus. Also see \cite{BB1} by Berti and
Biasco. Recently, the long time stability of KAM tori for nonlinear
Schr$\ddot{\mbox{o}}$dinger equation was obtained in \cite{CLY}. The
basic idea is that, first, define a suitable $p$-tame property which
generalizes the idea in \cite{BG}; second, prove the $p$-tame
property persists under KAM iterative procedure and normal form
iterative procedure; third, based on some reasonable non-resonant
conditions, one can construct a partial normal form of high order
around a KAM torus; finally, combining the $p$-tame property with
the partial normal form of high order, one can deduce that the
solutions starting in the neighborhood of a KAM torus still stay in
the neighborhood of the KAM torus for a polynomial long time, that
is, the KAM tori are stable in a long time.

In this paper, we will prove the long time stability of KAM tori for
nonlinear wave equation. The method follows a parallel course as in
\cite{CLY} except two essential differences: (1) more regularity is
needed for the Hamiltonian vector field generated by nonlinear
wave equation to guarantee that the KAM iterative procedure works
(it is necessary for measure estimate and see \cite{P1} for the
details); (2) since the frequencies of nonlinear wave equation has
worse approximations than nonlinear Schr$\ddot{\mbox{o}}$dinger
equation, whether the non-resonant conditions hold when constructing
the partial normal form of higher order? We will deal with the first
problem by modifying the definition of $p$-tame norm (see
(\ref{091402}) in Definition \ref{071803}) and estimate the measure
of non-resonant set by the method as in \cite{Bam2011} (see Section
5 for the details). The following is our main result:

\begin{thm}\label{T4}Consider the nonlinear
wave equation
\begin{equation}\label{26}
u_{tt}=u_{xx}-(m+M_{\xi})u+\varepsilon u^3,
\end{equation}
subject to Dirichlet boundary conditions $u(t,0)=u(t,\pi)=0$, where
$m$ is a non-negative constant and $M_{\xi}$ is a real Fourier
multiplier,
\begin{equation}\label{091913}
M_{\xi} \sin jx=\xi_j\sin jx,
\end{equation}
\begin{equation*}
\xi\in\Pi:=\left\{\xi=(\xi_j)_{j\geq1}|\ {\xi}_j\in [1,2]/j,\ j\geq
1\right\}\subset\mathbb{R}^{\mathbb{N}}.
\end{equation*}
Given an integer $n\geq1$ and a real number $p\geq1$, for any
sufficiently small $\varepsilon>0$, there exists a large subset $
\tilde\Pi\subset\Pi$, such that for each $\xi\in\tilde\Pi$ equation
(\ref{26}) possesses a linearly stable $n$-dimensional KAM torus
$\mathcal{T}_{\xi}$ in Sobolev space $H^p_0([0,\pi])$. Moreover, for
arbitrarily given $\mathcal M$ with $0\leq \mathcal{M}\leq
C(\varepsilon)$ (where $C(\varepsilon)$ is a constant depending on
$\varepsilon$ and $C(\varepsilon)\rightarrow\infty$ as
$\varepsilon\rightarrow0$) and $p\geq 24(\mathcal{M}+7)^{4}+1$,
there exists a small positive $\delta_0$ depending on $n,p$ and
$\mathcal{M}$, and for any $0<\delta<\delta_0$ and any solution
$u(t,x)$ of equation (\ref{26}) with the initial datum satisfying
$${d}_{H^p_0[0,\pi]}(u(0,x),\mathcal{T}_{\xi}):=\inf_{w\in\mathcal{T}_\xi}||u(0,x)-w||_{H^p_0[0,\pi]}\leq \delta,$$
then
\begin{equation*}{d}_{H^p_0[0,\pi]}(u(t,x),\mathcal{T}_{\xi}):=\inf_{w\in\mathcal{T}_\xi}||u(t,x)-w||_{H^p_0[0,\pi]}\leq
2\delta,\qquad \mbox{for all} \ |t|\leq {\delta}^{-\mathcal{M}}.
\end{equation*}
\end{thm}

\begin{rem}
Instead of equation (\ref{26}), we can also prove the long time stability of KAM tori for general nonlinear
wave equations, such as
\begin{equation}\label{082003}
u_{tt}=u_{xx}-V(x)u+\varepsilon g(x,u),
\end{equation}
where $V(x)$ is a smooth and $2\pi$ periodic potential, and
$g(x,u)$ is a smooth function on the domain
$\mathbb{T}\times\mathcal{U}$, $\mathcal{U}$ being a neighborhood of
the origin in $\mathbb{R}$.
 Equation (\ref{082003}) was discussed in \cite{BG} and shown that
 the origin is stable in long time by the infinite dimensional
 Birkhoff normal form theorem.
\end{rem}


The rest of the present paper is organized as follows. In \S 2, we
give some basic notations and the definition of $p$-tame norm for a
Hamiltonian vector field. In \S 3, we construct a normal form of
order 2, which satisfies $p$-tame property, around the KAM tori
based on the standard KAM method (see Theorem \ref{T1}) and a
partial normal form of order $\mathcal{M}+2$ in the
$\delta$-neighborhood of the KAM tori (see Theorem \ref{thm7.1}).
Since the iterative procedure is parallel to \cite{CLY}, we only
prove the measure estimate in detail. Finally, due to the partial
normal form of order $\mathcal{M}+2$ and $p$-tame property, we show
that the KAM tori are stable in a long time (see Theorem \ref{T3}).
In \S 4, we finish the proof of Theorem \ref{T4}. In \S 5, we list
some properties of $p$-tame norm. These properties are used in the
proof of Theorem \ref{T1} and Theorem \ref{thm7.1} to ensure the
$p$-tame property surviving under KAM iterative procedure and normal
form iterative procedure.

\section{Definition of $p$-tame norm for a Hamiltonian vector field}
We will define $p$-tame norm for a Hamiltonian vector field in this section. First we introduce the
functional setting and the main notations concerning infinite dimensional Hamiltonian systems. Consider the Hilbert space of complex-valued sequences
\begin{equation*}
\ell^2_p:=\left\{q=(q_1,q_2,\dots):||q||_p^2:=\sum_{j\geq1}|q_j|^2j^{2p}<+\infty\right\}
\end{equation*}
with $p>1/2$, and the symplectic phase space
\begin{equation*}
(x,y,z)\in
D(s)\times\mathbb{C}^n\times\ell_{b,p}^2:=\mathcal{P}^p,\qquad
z:=(q,\bar q)\in\ell_{b,p}^2:=\ell_p^2\times\ell^{2}_p,
\end{equation*}
where $D(s):=\left\{x\in\mathbb{C}^n/(2\pi\mathbb{Z})^n| \
||\mbox{Im}\ x||< s\right\}$ is the complex open $s$-neighborhood of
the $n$-torus $\mathbb{T}^n:=\mathbb{R}^n/(2\pi\mathbb{Z})^n$,
equipped with the canonic symplectic structure:
\begin{equation*}
\sum_{j=1}^{n}dy_j\wedge dx_j+\sqrt{-1}\sum_{j\geq1}dq_j\wedge d\bar{q}_j.
\end{equation*}
Let
\begin{equation*}
D(s,r_1,r_2)=\left\{(x,y,z)\in \mathcal{P}^p\left|\right.||\mbox{Im}\ x||< s,||y||< r_1^2,||z||_p< r_2\right\},
\end{equation*}
where $||\cdot||$ denote the sup-norm for
complex vectors and
\begin{equation*}
||z||_{p}=||q||_{p}+||\bar q||_p\qquad\mbox{with}\ z=(q,\bar q).
\end{equation*}
 Any analytic function $W:D(s,r_1,r_2)\rightarrow\mathbb{C}$ can be developed in a totally convergent power series
\begin{equation*}
W(x,y,z)=\sum_{\alpha\in\mathbb{N}^n,\beta\in\mathbb{N}^{\bar{\mathbb{Z}}}}
W^{\alpha\beta}(x)y^{\alpha}z^{\beta},\qquad
\bar{\mathbb{Z}}=\mathbb{Z}\setminus\{0\}.
\end{equation*}Note that there is a multilinear, symmetric, bounded map
\begin{equation*}
\widetilde{W^{\alpha\beta}(x)}\in\mathcal{L}\left(\overbrace{\mathbb{C}^n\times\dots\times\mathbb{C}^n}^{|\alpha|-times}\times\overbrace{
\ell_{b,p}^2\times\dots\times\ell^{2}_{b,p}}^{|\beta|-times},\mathbb{C}\right)
\end{equation*}
such that
$$
\widetilde {W^{\alpha\beta}(x)}(\overbrace{y,\dots,y}^{|\alpha|-times},
\overbrace{z.\dots,z}^{|\beta|-times})=W^{\alpha\beta}(x)y^{\alpha}z^{\beta},$$
where $|\alpha|=\sum_{j=1}^{n}|\alpha_j|, |\beta|=\sum_{j\in\bar{{\mathbb{Z}}}}|\beta_j|$, and $|\cdot|$ denotes the $1$-norm here and below.

We will study the Hamiltonian system
\begin{equation*}
(\dot{x},\dot{y},\dot{z})=X_W(x,y,z),
\end{equation*}
where $X_W$ is the Hamiltonian vector field of $W$,
\begin{equation*}
X_W=(W_y,-W_x,\sqrt{-1}JW_z),
\end{equation*}
and
\begin{equation*}
J:=\left(\begin{array}{cc}
0&I\\-I&0
\end{array}\right).
\end{equation*}

\begin{defn}\label{112901}
Consider a function $W(x;\xi):D(s)\times\Pi\rightarrow \mathbb{C}$
is analytic in the variable $x\in D(s)$ and $C^1$-smooth in the parameter $\xi\in\Pi$ in the
Whitney's sense\footnote{In the whole of this paper, the derivatives
with respect to the parameter $\xi\in\Pi$ are understood in the
sense of Whitney.}, and the Fourier
series of $W(x;\xi)$ is given by
$$W(x;\xi)=\sum_{k\in \mathbb Z^n}\widehat
W(k;\xi)e^{\sqrt{-1}\langle k,x\rangle},$$where $$\widehat
W(k;\xi):=\frac1{(2\pi)^n}\int_{\mathbb{T}^n}W(x;\xi)e^{-\sqrt{-1}\langle
k,x\rangle}dx$$ is the $k$-th Fourier coefficient of
$W(x;\xi)$, and $\langle\cdot,\cdot\rangle$ denotes the usual inner
product, i.e.
\begin{equation*}
\langle k,x\rangle=\sum_{j=1}^nk_jx_j.
\end{equation*}
Then define the norm $||\cdot||_{D(s)\times\Pi}$ of $W(x;\xi)$ by
\begin{equation}\label{092704}
||W||_{D(s)\times\Pi}=\sup_{\xi\in\Pi,j\geq 1}{\sum_{k\in\mathbb{Z}^n}
\left(|\widehat{W}(k;\xi)|+|\partial_{\xi_j}
\widehat{W}(k;\xi)|\right)e^{|k|s}}.
\end{equation}
\end{defn}
\begin{defn}\label{112902}Let
\begin{equation*}
D(s,r)=\{(x,y)\in D(s)\times\mathbb{C}^n|\ ||\mbox{Im}\ x||< s,\
||y||< r^2\}.
\end{equation*}Consider a function $W(x,y;\xi):D(s,r)\times\Pi\rightarrow
\mathbb{C}$ is analytic in the variable $(x,y)\in D(s,r)$ and
$C^1$-smooth in the parameter $\xi\in\Pi$ with the following form
\begin{equation*}
W(x,y;\xi)=\sum_{\alpha\in\mathbb{N}^n}W^{\alpha}(x;\xi)y^{\alpha}.
\end{equation*}Then
define the norm $||\cdot||_{D(s,r)\times\Pi}$ of $W(x,y;\xi)$
by \begin{equation}\label{092703}
||W||_{D(s,r)\times\Pi}=\sum_{\alpha\in\mathbb{N}^n}
|||\widetilde{\mathcal{W}^{\alpha}}|||\;r^{2|\alpha|},
\end{equation}
where $\mathcal{W}^{\alpha}=||W^{\alpha}(x;\xi)||_{D(s)\times\Pi}$,
 $\widetilde{{\mathcal W}^{\alpha}}\in\mathcal{L}(\overbrace{\mathbb{C}^n\times\dots\times\mathbb{C}^n}^{|\alpha|-times},\mathbb{C})$$ $ is an $|\alpha|$-linear symmetric bounded map such that $$
\widetilde {\mathcal W^{\alpha}}(\overbrace{y,\dots,y}^{|\alpha|-times})=\mathcal{W}^{\alpha}y^{\alpha},$$and $|||\cdot|||$ is the operator norm of multilinear symmetric bounded maps.
\end{defn}

\begin{defn}\label{021002}
Consider a function
$W(x,y,z;\xi):D(s,r,r)\times\Pi\rightarrow\mathbb{C}$ is
analytic in the variable $(x,y,z)\in D(s,r,r)$ and
$C^1$-smooth in the parameter $\xi\in\Pi$ with the following form
$$W(x,y,z;\xi)=\sum_{\beta\in\mathbb{N}^{\bar{\mathbb{Z}}}}
W^{\beta}(x,y;\xi)z^{\beta}.$$
Define the modulus $\lfloor W\rceil_{D(s,r)\times\Pi}(z)$ of
$W(x,y,z;\xi)$ by
\begin{equation}\label{081602}\lfloor W\rceil_{D(s,r)\times\Pi}(z):=
\sum_{\beta\in\mathbb{N}^{\bar{\mathbb{Z}}}}||W^{\beta}||_{D(s,r)\times\Pi}\;
z^{\beta}.\end{equation}
\end{defn}
\begin{defn}\label{071803}Let $$W(x,y,z;\xi):=W_h(x,y,z;\xi)=\sum_{\beta\in\mathbb{N}^{\bar{\mathbb{Z}}},|\beta|=h}
W_h^{\beta}(x,y;\xi)z^{\beta}$$ be a function
is
analytic in the variable $(x,y,z)\in D(s,r,r)$ and
$C^1$-smooth in the parameter $\xi\in\Pi$, and let
\begin{eqnarray}\label{091401}||(z^{h})||_{p,1}:=
\frac{1}{h}\sum_{j=1}^{h}||z^{(1)}||_1\cdots||z^{(j-1)}||_1||
z^{(j)}||_p||z^{(j+1)}||_1\cdots||z^{(h)}||_1.
\end{eqnarray}
Define the $p$-tame operator norm for $W_z$ by
\begin{eqnarray}\label{091402}
|||{W_z}|||_{p,D(s,r)\times\Pi}^{T}: =\sup_{0\neq z^{(j)}\in \ell
^2_{b,p},1\leq j\leq h-1}
\frac{||{\lfloor{\widetilde{JW_z}}\rceil}_{D(s,r)\times\Pi}(z^{(1)},\dots,z^{(h-1)})||_
{p+1}} {||(z^{h-1})||_{p,1}},\end{eqnarray}
define the $1$-operator norm for $W_z$ by
\begin{eqnarray}\label{122401}
|||{W_z}|||_{1,D(s,r)\times\Pi}: =\sup_{0\neq z^{(j)}\in \ell
^2_{b,1},1\leq j\leq h-1}
\frac{||{\lfloor{\widetilde{JW_z}}\rceil}_{D(s,r)\times\Pi}(z^{(1)},\dots,z^{(h-1)})||_
{1}} {||(z^{h-1})||_{1,1}},\end{eqnarray} and define the operator norm for $W_v$ ($v=x$ or $y$) by
\begin{eqnarray}\label{091404} |||W_{v}|||_{D(s,r)\times\Pi}:=\sup_{0\neq
z^{(j)}\in \ell ^2_{b,1},1\leq j\leq h}
\frac{||{\lfloor{\widetilde{W_{v}}}\rceil}_{D(s,r)\times\Pi}(z^{(1)},\dots,z^{(h)})||}
{||(z^{h})||_{1,1}}.
\end{eqnarray}Finally define the $p$-tame norm of the Hamiltonian vector field $X_W$ as
follows,
\begin{eqnarray}\label{051703}
|||X_W|||_{p,D(s,r,r)\times\Pi}^T =|||{W_y}|||_{D(s,r,r)\times\Pi}
+\frac1{r^2}|||{W_x}|||_{D(s,r,r)\times\Pi}
+\frac1r|||W_z|||_{p,D(s,r,r)\times \Pi}^T,
\end{eqnarray}where
\begin{eqnarray}\label{091405}
|||{W_v}|||_{D(s,r,r)\times\Pi}:=|||{W_v}|||_{D(s,r)\times\Pi}r^{h},\qquad v=x\ \mbox{or}\ y,
\end{eqnarray}
and
\begin{eqnarray}\label{091403}
|||{W_z}|||_{p,D(s,r,r)\times\Pi}^T=\max
\left\{|||{W_z}|||_{p,D(s,r)\times\Pi}^T,|||{W_z}|||_{1,D(s,r)\times\Pi}\right\}r^{h-1}.
\end{eqnarray}
\end{defn}
\begin{rem}
In view of (\ref{091402}), $\lfloor{{JW_z}}\rceil_{D(s,r)\times\Pi}$ is required as a bounded map form $\ell^{2}_{b,p}$ to $\ell^{2}_{b,p+1}$ instead of a bounded map form $\ell^{2}_{b,p}$ to $\ell^{2}_{b,p}$ as in \cite{CLY}.
\end{rem}\begin{rem}
Based on (\ref{091402}) and (\ref{091404})
 in Definition \ref{071803}, for each
$(x,y,z)\in\mathcal{P}^p$ and $\xi\in\Pi$, the following estimates hold
\begin{equation}\label{011602}
||(W_h)_z(x,y,z;\xi)||_p\leq||(W_h)_z(x,y,z;\xi)||_{p+1}\leq
|||{(W_h)_z}|||_{p,D(s,r)\times\Pi}^{T}||z||_p||z||_1^{\max\{h-2,0\}},
\end{equation}
\begin{equation}\label{012201} ||(W_h)_x(x,y,z;\xi)||\leq
|||{(W_h)_x}|||_{D(s,r)\times\Pi}||z||_1^{h},
\end{equation}
and
\begin{equation}\label{012202} ||(W_h)_y(x,y,z;\xi)||\leq
|||{(W_h)_y}|||_{D(s,r)\times\Pi}||z||_1^{h}.
\end{equation}
\end{rem}
\begin{defn}\label{080204} Let $W(x,y,z;\xi)=\sum_{h\geq
0}W_{h}(x,y,z;\xi)$ be a Hamiltonian analytic in the variable $(x,y,z)\in
D(s,r,r)$ and $C^1$-smooth in the parameter $\xi\in\Pi$, where
$$W_{h}(x,y,z;\xi)=\sum_{\beta\in\mathbb{N}^{\bar{\mathbb{Z}}},|\beta|=h}W_{h}^{\beta}(x,y;\xi)
z^{\beta}.$$ Then define the $p$-tame norm of the Hamiltonian
vector field $X_W$ by
\begin{equation}\label{102201}
|||X_W|||_{p,D(s,r,r)\times\Pi}^T:=\sum_{h\geq 0}
|||X_{W_{h}}|||_{p,D(s,r,r)\times\Pi}^T. \end{equation} Moreover, we say that
a Hamiltonian vector field $X_W$ (or a Hamiltonian $W(x,y,z;\xi)$)
has ${p}$-tame property on the domain $D(s,r,r)\times\Pi$, if and
only if $ |||X_W|||_{p,D(s,r,r)\times\Pi}^{T}<\infty. $
\end{defn}

\section{The abstract results}

As in \cite{P1}, define \begin{equation}\label{032601}
||w||_{\mathcal{P}^p,D(s,r,r)}=||x||+\frac1{r^2}||y||+\frac1{r}||z||_p
 \end{equation}for each $w=(x,y,z)\in D(s,r,r)$, and
 define the usual weighted norm of Hamiltonian vector field $X_U$ on the domain
$D(s,r,r)\times\Pi$ by
\begin{equation}\label{081101}
|||X_U|||_{\mathcal{P}^p,D(s,r,r)\times\Pi}=\sup_{(x,y,z;\xi)\in
D(s,r,r)\times\Pi}\left(||U_y||+\frac1{r^2}||U_x||+\frac1r||U_z||_{p+1}\right).
\end{equation}

\begin{thm}\label{T1}(Normal form of order 2)
Consider a perturbation of the integrable Hamiltonian
\begin{equation}\label{100803} H(x,y,q,\bar q;\xi)=N(y,q,\bar q;\xi)+R(x,y,q,\bar
q;\xi)
\end{equation}defined on the domain $D(s_0,r_0,r_0)\times\Pi$ with
$s_0,r_0\in(0,1]$, where
$$N(y,q,\bar q;\xi)=\sum_{j=1}^{n}\omega_j(\xi)y_j+\sum_{j\geq1}\Omega_j(\xi)q_j\bar
q_j$$ is a family of parameter dependent integrable Hamiltonian and
$$R(x,y,q,\bar q;\xi)=\sum_{\alpha\in\mathbb{N}^n,\beta,\gamma\in\mathbb{N}^{\mathbb{N}}}R^{\alpha\beta\gamma}(x;\xi)y^{\alpha}q^{\beta}\bar q^{\gamma}$$ is the perturbation. Suppose the tangent
frequency and normal frequency satisfy the following assumption:

$Frequency\ Asymptotics.$
\begin{equation}\label{012801}
\omega_j(\xi)=\sqrt{j^2+m+\xi_j},\qquad \mbox{for} \ 1\leq j\leq n
\end{equation}
and
\begin{equation}\label{012802}
\Omega_j(\xi)=\sqrt{(j+n)^2+m+\xi_{j+n}}\qquad \mbox{for} \ j\geq1,
\end{equation}
where $\xi=(\xi_j)_{j\geq1}\in\Pi.$
The perturbation $R(x,y,q,\bar q;\xi)$ has $p$-tame property on the
domain $D(s_0,r_0,r_0)\times\Pi$ and satisfies the small assumption:
\begin{equation*}
 \varepsilon:=|||X_{R}|||^T_{p,D(s_0,r_0,r_0)\times\Pi}\leq
 \eta^{12}\epsilon, \qquad \mbox{ for some }\ \eta\in(0,1),
\end{equation*}where $\epsilon$ is a positive constant
depending on $s_0,r_0$ and $n$. Then there exists a subset
$\Pi_{\eta}\subset\Pi$ with the estimate
\begin{equation*}
\mbox{Meas}\ \Pi_{\eta}\geq(\mbox{Meas}\ \Pi)(1-O(\eta^{1/2})).
\end{equation*}For each $\xi\in\Pi_{\eta}$, there is a symplectic map
$$\Psi: D(s_0/2,r_0/2,r_0/2)\rightarrow D(s_0,r_0,r_0),$$ such that
\begin{equation}\label{081601}\breve H(x,y,q,\bar q;\xi):=H\circ\Psi=
\breve N(y,q,\bar q;\xi)+ \breve R(x,y,q,\bar q;\xi),
\end{equation}where
\begin{equation}\label{082903}
\breve N(y,q,\bar
q;\xi)=\sum_{j=1}^{n}\breve{\omega}_j(\xi)y_j+\sum_{j\geq1}\breve{\Omega}_j(\xi)q_j\bar
q_j
\end{equation}
and
\begin{equation}\label{082902}
\breve R(x,y,q,\bar
q;\xi)=\sum_{\alpha\in\mathbb{N}^n,\beta,\gamma\in\mathbb{N}^{\mathbb{N}},2|\alpha|+|\beta|+|\gamma|\geq
3}\breve
R^{\alpha\beta\gamma}(x;\xi)y^{\alpha}q^{\beta}\bar{q}^{\gamma}.
\end{equation}
Moreover, the following estimates hold:\\
(1) for each $\xi\in\Pi_{\eta},$ the symplectic map
$\Psi:D(s_0/2,r_0/2,r_0/2)\rightarrow D(s_0,r_0,r_0)$
\color{black}satisfies
\begin{equation}\label{080101}
||\Psi-id||_{p,D(s_0/2,r_0/2,r_0/2)}\leq c\eta^6\epsilon
\end{equation}
and
\begin{equation}\label{080102}
|||D\Psi-Id|||_{p,D(s_0/2,r_0/2,r_0/2)}\leq c\eta^6\epsilon,
\end{equation}
where on the left-hand side hand we use the operator
norm\footnote{where $id$ denotes the identity map from $\mathcal
P^p\to\mathcal P^p$ and $Id$ denotes its tangent map. }
\begin{equation*}
|||D\Psi-Id|||_{p,D(s_0/2,r_0/2,r_0/2)}=\sup_{0\neq w\in
D(s_0/2,r_0/2,r_0/2) }\frac{||(D\Psi-Id) w||_{\mathcal{P}^p,D(s_0,r_0,r_0) }}{||
w||_{\mathcal{P}^p,D(s_0/2,r_0/2,r_0/2)}};
\end{equation*}
\\
\color{black} (2) the frequencies $\breve \omega(\xi)$ and $\breve
\Omega(\xi )$ satisfy
\begin{equation}\label{080103}
||\breve\omega(\xi)-\omega(\xi)||+\sup_{j\geq
1}||\partial_{\xi_j}(\breve\omega(\xi)-\omega(\xi))||\leq
{c\eta^8\epsilon}\end{equation}and
\begin{equation}\label{080104}
||\breve\Omega(\xi)-\Omega(\xi)||_{-1}+\sup_{j\geq
1}||\partial_{\xi_j}(\breve\Omega(\xi)-\Omega(\xi))||_{-1}\leq
{c{\eta^8}\epsilon},
\end{equation}
where
\begin{equation}
||w=(w_i)_{i\geq1}||_{-1}:=\sup_{i\geq1}|w_i(\xi)i|;
\end{equation}
(3) the Hamiltonian vector field $X_{\breve R}$ of the new perturbed
Hamiltonian $\breve R(x,y,q,\bar q;\xi)$ satisfies
\begin{equation}\label{080105}
 |||X_{\breve R}|||^T_{p,D(s_0/2,r_0/2,r_0/2)\times\Pi_{\eta}}\leq
 \varepsilon(1+{c\eta^6\epsilon}),
\end{equation}where $c>0$ is a constant depending on $s_0,r_0$ and $n$.
\end{thm}

\begin{rem}
This theorem is parallel to Theorem 2.9 in \cite{CLY} and is
essentially due to a standard KAM proof. The same as in \cite{CLY},
the tame property (\ref{080105}) of $X_{\breve R}$ can be verified
explicitly in view of Lemmas \ref{021102}-\ref{081503}. Moreover, as
a corollary of this theorem, the existence and time $\delta^{-1}$
stability can be obtained directly ($p$-tame property is not
necessary here).
\end{rem}

Starting from the normal form of order 2 obtained in Theorem
\ref{T1}, we will further construct a partial normal form of order
$\mathcal{M}+2$ through $\mathcal{M}$-times symplectic
transformations under some non-resonant conditions. To this end,
some notations are given first. Given a large
$\mathcal{N}\in\mathbb{N}$, split the normal frequency $\Omega(\xi)$
and normal variable $(q,\bar q)$ into two parts respectively, i.e.
$$\Omega(\xi)=(\tilde {\Omega}(\xi),\hat{\Omega}(\xi)),\qquad q=(\tilde
q,\hat q),\qquad \bar q=(\bar{\tilde q},\bar {\hat q}),$$ where
$$\tilde
{\Omega}(\xi)=(\Omega_1(\xi),\dots,\Omega_{\mathcal{N}}(\xi)),\quad
\tilde q=(q_1,\dots,q_{\mathcal{N}}),\quad \bar{\tilde q}=(\bar
q_{1},\dots,\bar q_{\mathcal{N}})$$ are the low frequencies  and
$$
\hat{\Omega}(\xi)=(\Omega_{\mathcal{N}+1}(\xi),\Omega_{\mathcal{N}+2}(\xi),\dots),\quad
\hat q=(q_{\mathcal{N}+1},q_{\mathcal{N}+2},\dots),\quad \bar{\hat
q}=(\bar q_{\mathcal{N}+1},\bar q_{\mathcal{N}+2},\dots)$$
 are the high frequencies. Given $0<\tilde \eta<1$, and $\tau>
 2n+5$, define the
resonant sets $\mathcal{R}_{k\tilde l\hat l}$ by
\begin{equation}\label{091202}
\mathcal{R}_{k\tilde{l}\hat{l}}=\left\{\xi\in\Pi_{\eta}:\left|\langle
k,\breve\omega(\xi)\rangle+\langle \tilde
l,\tilde{\Omega}(\xi)\rangle+\langle
\hat{l},\hat{\Omega}(\xi)\rangle\right|\leq\frac{\tilde\eta}{4^{3\mathcal{M}}(|
k|+1)^{\tau}C(\mathcal{N},\tilde l)}\right\}.
\end{equation}
Let
\begin{equation}\label{091208}
\mathcal{R}=\bigcup_{|k|+|\tilde l|+|\hat l|\neq 0, |\tilde l|+|\hat
l|\leq \mathcal{M}+2, |\hat l|\leq 2}\mathcal{R}_{k\tilde{l}\hat{l}}
\end{equation}and
\begin{equation}\label{091215}
\tilde\Pi_{\tilde\eta}=\Pi_{\eta}\setminus\mathcal{R},
\end{equation}
where $\Pi_{\eta}$ is defined in Theorem \ref{T1}.


\begin{thm}\label{thm7.1} (Partial normal form of order $\mathcal{M}+2$)
Consider the normal form of order 2 $$\breve H(x,y,q,\bar
q;\xi)=\breve N(y,q,\bar q;\xi)+\breve R(x,y,q,\bar q;\xi)$$
obtained in Theorem \ref{T1}. Given any positive integer
$\mathcal{M}$ and $0<\tilde \eta<1$, there exist a small $\rho_0>0$
and a large positive integer $\mathcal{N}_0$ depending on
$s_0,r_0,n$ and $\mathcal{M}$, such that for each $0<\rho<\rho_0$, any
integer $\mathcal{N}$ satisfying
\begin{equation}\label{091113}\mathcal{N}_0<\mathcal{N}<\left(\frac{\tilde\eta^2}{2\rho}\right)^{\frac{1}{6(\mathcal{M}+7)^2}}, \end{equation}
the non-resonant set $\tilde\Pi_{\tilde\eta}$ fulfills the estimate
\begin{equation}\label{091301} Meas\ \tilde\Pi_{\tilde\eta} \geq(Meas\ \Pi_{\eta})(1-c\tilde\eta^{1/2}),
\end{equation}
and for any $\xi\in\tilde\Pi_{\tilde\eta}$, there is a symplectic
map
$$\Phi: D(s_0/4,4\rho,4\rho)\rightarrow D(s_0/2,5\rho,5\rho),$$ such
that
\begin{equation}\label{020302}\breve{\breve{H}}(x,y,q,\bar q;\xi):=
\breve H\circ\Phi={\breve N}(y,q,\bar q;\xi)+{Z}(y,q,\bar q;\xi)+{
P}(x,y,q,\bar q;\xi)+{ Q}(x,y,q,\bar q;\xi)
\end{equation}
is a partial normal form of order $\mathcal{M}+2$, where
\begin{eqnarray*}
{ Z}(y,q,\bar q;\xi)&=&\sum_{4\leq 2|\alpha|+2|\beta|+2|\mu|\leq
\mathcal{M}+2,|\mu|\leq1}{
Z}^{\alpha\beta\beta\mu\mu}(\xi)y^{\alpha}\tilde{q}^{\beta}\bar{\tilde
q}^{\beta}\hat{q}^{\mu}\bar{\hat{q}}^{\mu}
\end{eqnarray*}is the integrable term depending only on $y$ and $I_j=|q_j|^2,j\geq1$, and where
\begin{eqnarray*}  {
P}(x,y,q,\bar
q;\xi)&=&\sum_{2|\alpha|+|\beta|+|\gamma|+|\mu|+|\nu|\geq
\mathcal{M}+3,|\mu|+|\nu|\leq 2} { P}^{\alpha
\beta\gamma\mu\nu}(x;\xi)y^{\alpha}{\tilde q}^{\beta}{\bar{\tilde
q}}^{\gamma}{\hat q}^{\mu}{\bar{\hat q}^{\nu}}
\end{eqnarray*}and
\begin{equation*}
{ Q}(x,y,q,\bar q;\xi)=\sum_{|\mu|+|\nu|\geq 3}{ Q}^{\alpha
\beta\gamma\mu\nu}(x;\xi)y^{\alpha}{\tilde q}^{\beta}{\bar{\tilde
q}}^{\gamma}{\hat q}^{\mu}{\bar{\hat q}}^{\nu}.
\end{equation*}
Moreover, we have the following estimates:\\
(1) the symplectic map $\Phi$ satisfies
\begin{equation}\label{091812}
||\Phi-id||_{p,D(s_0/4,4\rho,4\rho)}\leq \frac{c\mathcal{N}^{294}\rho}{\tilde\eta^2}
\end{equation}
and
\begin{equation}
|||D\Phi-Id|||_{p,D(s_0/4,4\rho,4\rho)}\leq \frac{c\mathcal{N}^{294}}{\tilde\eta^2};
\end{equation}
\\
(2) the Hamiltonian vector fields $X_Z,X_P$ and $X_Q$ satisfy
\begin{equation*}
|||X_{{ Z}}|||^T_{p,D(s_0/4,4\rho,4\rho)\times\Pi_{\eta}}\leq
c\rho
\left(\frac1{\tilde\eta^2}\mathcal{N}^{6(\mathcal{M}+6)^2}\rho\right),
\end{equation*}
\begin{equation}\label{091304}
|||X_{{ P}}|||^T_{p,D(s_0/4,4\rho,4\rho)\times\Pi_{\eta}}\leq
c
\rho\left(\frac1{\tilde\eta^2}\mathcal{N}^{6(\mathcal{M}+7)^2}\rho\right)^{\mathcal{M}}\qquad
 \end{equation}and
\begin{equation*}
|||X_{{ Q}}|||^T_{p,D(s_0/4,4\rho,4\rho)\times\Pi_{\eta}}\leq
c \rho,
\end{equation*}
where $c>0$ is a constant depending on $s_0,r_0,n$ and
$\mathcal{M}$.
\end{thm}

\begin{proof}
In this proof, Lemmas \ref{021102}-\ref{081503} are used, and the
normal form iterative procedure is the same as Theorem 5.1 in
\cite{CLY}. Thus, we prove the measure estimate (\ref{091301}) in
detail while omit the other parts of proof. Firstly, we will
estimate the measure of the resonant sets $\mathcal{R}_{k\tilde
l\hat l}$.
\\$\textbf{Case 1.}$\\
For $|k|\neq0$, without loss of generality, we assume
\begin{equation}\label{091206}
|k_1|=\max_{1\leq i\leq n}\{|k_1|,\dots,|k_n|\}. \end{equation} Then
\begin{eqnarray*}
&&|\partial_{\xi_1}(\langle k,\breve\omega(\xi)\rangle+\langle
\tilde
l,\tilde{\Omega}(\xi)\rangle+\langle\hat l,\hat{\Omega}(\xi)\rangle)|\\
&\geq&
|k_1||\partial_{\xi_1}\breve\omega_1(\xi)|-|\partial_{\xi_1}(\sum_{i=2}^nk_i\breve\omega_i(\xi)+\langle
\tilde l,\tilde{\Omega}(\xi)\rangle+\langle\hat l,\hat{\Omega}(\xi)\rangle)|\\
&\geq&|k_1|(1-c\eta^8\epsilon)-(\sum_{i=2}^n|k_i|+|\tilde l|+|\hat l|)c\eta^{8}\epsilon\qquad \mbox{(in view of (\ref{012801}), (\ref{012802}), (\ref{080103}) and (\ref{080104}))}\\
&\geq&|k_1|-(|k|+\mathcal{M}+2)c\eta^{8}\epsilon\qquad \mbox{(in
view of $|\tilde l|+|\hat l|\leq \mathcal{M}+2$)}\\&\geq&
\frac14|k_1|\qquad\qquad\qquad \mbox{(by (\ref{091206}) and
$\mathcal{M}\leq(2c\eta^{8}\epsilon)^{-1})$}\\
&\geq&\frac14.
\end{eqnarray*}
Hence,
\begin{equation}\label{101102}
Meas\ \mathcal{R}_{k\tilde l\hat
l}\leq\frac{4\tilde\eta}{4^{3\mathcal{M}}(|
k|+1)^{\tau}C(\mathcal{N},\tilde l)}\cdot Meas\ \Pi_{\eta}.
\end{equation}
$\textbf{Case 2.}$ \\
If $|k|=0$ and $|\tilde l|\neq0$, without loss of generality, we
assume \begin{equation*}|\tilde {l}_j|\neq 0
\end{equation*}
and
\begin{equation*}
\tilde{l}_i=0,\qquad 1\leq i\leq j-1.
\end{equation*}
Then
\begin{eqnarray*}
&&|\partial_{\xi_{n+j}}(\langle k,\breve\omega(\xi)\rangle+\langle
\tilde
l,\tilde{\Omega}(\xi)\rangle+\langle\hat l,\hat{\Omega}(\xi)\rangle)|\\
&\geq& |\tilde
l_j||\partial_{\xi_{n+j}}\tilde{\Omega}_j(\xi)|-|\partial_{\xi_{n+j}}(\langle
\tilde l,\tilde{\Omega}(\xi)\rangle+\langle\hat l,\hat{\Omega}(\xi)\rangle-\tilde l_j\tilde{\Omega}_j(\xi))|\\
&\geq&\left(\sqrt{(j+n)^2+m+\xi_j}\right)^{-1}\left(|\tilde l_j|(1-c\eta^8\epsilon)-\left(\sum_{i=j+1}^{\mathcal{N}}|\tilde l_i|+|\hat l|\right)c\eta^8\epsilon\right)\qquad\\
&& \mbox{(by (\ref{012802}) and (\ref{080104}))}\\
&\geq&\left(\sqrt{(j+n)^2+m+\xi_j}\right)^{-1}\left(|\tilde l_j|-(\mathcal{M}+2)c\eta^8\epsilon\right)\qquad\mbox{(in view of $|\tilde l|+|\hat l|\leq\mathcal{M}+2$)}\\
&\geq& \frac1{\mathcal{N}}|\tilde l_j|\qquad\qquad\qquad\qquad\qquad
\mbox{(in
view of $\mathcal{M}\leq(2c\eta^{8}\epsilon)^{-1}$ and $j\leq \mathcal{N}$)}\\
&\geq&\frac1{\mathcal{N}}.
\end{eqnarray*}
Hence,
\begin{equation}\label{101103}
Meas\ \mathcal{R}_{0\tilde l\hat
l}\leq\frac{4\tilde\eta\mathcal{N}}{4^{3\mathcal{M}}C(\mathcal{N},\tilde
l)}\cdot Meas \ \Pi_{\eta}.
\end{equation}
$\textbf{Case 3.}$ \\If $|k|=0,|\tilde l|=0$ and
$1\leq|\hat{l}|\leq2$, then it is easy to see that $ |\langle \hat
l, \hat \Omega(\xi)\rangle|$ is not small, i.e.
\begin{equation}\label{101104}\mbox{the sets $\mathcal{R}_{k\tilde
l\hat l}$ are empty for $|k|=0,|\tilde l|=0$ and
$1\leq|\hat{l}|\leq2$.}\end{equation}

Now we would like to estimate the measure of $\mathcal{R}$ (see
(\ref{091208})). Following the notations in \cite{Bam2011}, we
define the set
\begin{equation*}
\mathcal{Z}_{n,\mathcal{N}}:=\left\{(k,\tilde l,\hat
l)\in\mathbb{Z}^n\times\mathbb{Z}^{\mathcal{N}}\times\mathbb{Z}^{\mathbb{N}}\setminus(0,0,0):|\hat
l|\leq 2\right\},
\end{equation*}
and we split
\begin{equation*}
\mathcal{L}:=\left\{\hat l\in\mathbb{Z}^{\mathbb{N}}:|\hat l|\leq
2\right\}
\end{equation*}
as the union of the following four disjoint sets
\begin{eqnarray*}
\mathcal{L}_0&=&\{\hat l=0\},\\
\mathcal{L}_1&=&\{\hat l=e_j\},\\
\mathcal{L}_{2+}&=&\{\hat l=e_i+e_j\},\\
\mathcal{L}_{2-}&=&\{\hat l=e_i-e_j,i\neq j\},
\end{eqnarray*}
where
\begin{equation*}
e_j=(0,\dots,0,\overbrace{1}^{j-th},0,\dots,)
\end{equation*}
and $i,j\geq n+\mathcal{N}+1$.

Let $|\hat l|=2$ and $\hat{l}=e_i+e_j\in{\mathcal{L}}_{2+}$ for some
$i,j\geq n+\mathcal{N}+1$. If $$\min\{i,j\}\geq
|k|\cdot||\breve{\omega}({\xi})||+2(\mathcal{M}+2)\mathcal{N}+1,$$
 then it is easy to see that
\begin{equation*}
\left|\langle k,\breve\omega(\xi)\rangle+\langle \tilde
l,\tilde{\Omega}(\xi)\rangle+\langle
\hat{l},\hat{\Omega}(\xi)\rangle\right|\geq 1,
\end{equation*}which is not small. Namely, the resonant sets $\mathcal{R}_{k\tilde l\hat l}$ is empty. So it is sufficient to consider $$\max\{i,j\}< |k|\cdot||\breve{\omega}({\xi})||+2(\mathcal{M}+2)\mathcal{N}+1,$$
when the estimate (\ref{022418}) is given below. In fact, we obtain
\begin{eqnarray}
&&\nonumber Meas \bigcup_{(k,\tilde l,\hat l)\in
\mathcal{Z}_{n,\mathcal{N}}\bigcap{\mathcal{L}}_{2+}}\mathcal{R}_{k\tilde
l\hat l}\\&\leq& \nonumber\sum_{k\neq 0, (k,\tilde l,\hat l)\in
\mathcal{Z}_{n,\mathcal{N}}\bigcap{\mathcal{L}}_{2+}}
\frac{4\tilde\eta}{4^{3\mathcal{M}}(|
k|+1)^{\tau}C(\mathcal{N},\tilde l)}\cdot Meas\ \Pi_{\eta}\\
&&\nonumber+\sum_{k=0,(k,\tilde l,\hat l)\in
\mathcal{Z}_{n,\mathcal{N}}\bigcap{\mathcal{L}}_{2+}}
\frac{4\tilde\eta\mathcal{N}}{4^{3\mathcal{M}}C(\mathcal{N},\tilde
l)}\cdot Meas \ \Pi_{\eta}\\&\leq &c\tilde \eta \cdot Meas \
\Pi_{\eta},\label{022418}
\end{eqnarray}where $c_1>0$ is a constant depending on $n$ and $\tau.$

Similarly we obtain
\begin{equation}
Meas \bigcup_{(k,\tilde l,\hat l)\in
\mathcal{Z}_{n,\mathcal{N}}\bigcap{\mathcal{L}}_{0}}\mathcal{R}_{k\tilde
l\hat l}\leq c_2\tilde\eta \cdot Meas \ \Pi_{\eta}
\end{equation}
and
\begin{equation}
Meas \bigcup_{(k,\tilde l,\hat l)\in
\mathcal{Z}_{n,\mathcal{N}}\bigcap{\mathcal{L}}_{1}}\mathcal{R}_{k\tilde
l\hat l}\leq c_2\tilde\eta \cdot Meas \ \Pi_{\eta},
\end{equation}
where $c_2>0$ is a constant depending on $n$ and $\tau$. Now let
\begin{equation*}
(k,\tilde l,\hat l)\in
\mathcal{Z}_{n,\mathcal{N}}\bigcap{\mathcal{L}}_{2-},
\end{equation*}
and assume $i>j$ without loss generality. In view of (\ref{012802})
and (\ref{080104}), there is a constant $C>0$ such that
\begin{equation*}
\left|\frac{\breve \Omega_i(\xi)-\breve
\Omega_j(\xi)}{i-j}-1\right|\leq \frac C{j}.
\end{equation*}
Hence,
\begin{equation*}
\langle\hat
l,\hat\Omega(\xi)\rangle=\breve\Omega_i(\xi)-\breve\Omega_j(\xi)=i-j+r_{ij},
\end{equation*}
with
\begin{equation*}
|r_{ij}|\leq \frac {Cm}{j}
\end{equation*}
and $m=i-j$. Then we have
\begin{equation*}
\left|\langle k,\breve\omega(\xi)\rangle+\langle \tilde
l,\tilde\Omega(\xi)\rangle+\langle\hat
l,\hat\Omega(\xi)\rangle\right|\geq \left|\langle
k,\breve\omega(\xi)\rangle+\langle \tilde
l,\tilde\Omega(\xi)\rangle+m\right|-|r_{ij}|.
\end{equation*}
Therefore,
\begin{equation*}
\mathcal{R}_{k\tilde l\hat l}\subset \mathcal{Q}_{k\tilde l
mj}:=\left\{\left|\langle k,\breve\omega(\xi)\rangle+\langle \tilde
l,\tilde\Omega(\xi)\rangle+m\right|\leq
\frac{\tilde\eta}{4^{3\mathcal{M}}(|
k|+1)^{\tau}C(\mathcal{N},\tilde l)}+\frac{Cm}j\right\}.
\end{equation*}
For $j\geq j_0$, we have
\begin{equation*}
\mathcal{Q}_{k\tilde l mj}\subset \mathcal{Q}_{k\tilde l mj_0}.
\end{equation*}
Then it is sufficient to consider
\begin{equation*}
m\leq
|k|\cdot||\breve{\omega}({\xi})||+2(\mathcal{M}+2)\mathcal{N}+1,
\end{equation*}
and let
\begin{equation*}
j_0=\tilde
\eta^{-1/2}4^{\mathcal{M}}(|k|+1)^{\tau/2}C(\mathcal{N},\tilde
l)^{1/2}.
\end{equation*}
Then following the proof of Lemma 5 in \cite{Bam2011}, we obtain
\begin{equation}\label{022419}
Meas\ \bigcup_{(k,\tilde l,\hat l)\in
\mathcal{Z}_{n,\mathcal{N}}\bigcap{\mathcal{L}}_{2-}}\mathcal{R}_{k\tilde
l\hat l}\leq c_3\tilde\eta^{1/2} \cdot Meas\ \Pi_{\eta},
\end{equation}
where $c_3>0$ is a constant depending on $n$ and $\tau$. Finally, in
view of (\ref{022418})-(\ref{022419}) and (\ref{091208}), we obtain
\begin{equation}\label{022420}
Meas\ \mathcal{R} \leq c\tilde\eta^{1/2} \cdot Meas\ \Pi_{\eta},
\end{equation}
where $c$ is a constant depending on $n$ and $\tau$. Then combining
(\ref{091215}) with (\ref{022420}), we finish the proof of
(\ref{091301}).
\end{proof}

Based on the partial normal form of order $\mathcal{M}+2$ and
$p$-tame property, we obtain the long time stability of KAM tori as
follows.
\begin{thm}\label{T3}(The long time stability of KAM tori) Based on the partial normal form (\ref{020302}), for any $p\geq
24(\mathcal{M}+7)^{4}+1$ and $0<\delta<\rho$, the KAM tori ${\mathcal T}$ are stable in
long time, i.e. if $w(t)$ is a solution of Hamiltonian vector field
$X_{{H}}$ with the initial datum $w(0)=( w_x(0), w_y(0),
w_q(0),w_{\bar q}(0))$ satisfying
\begin{equation*}
d_{p}( w(0),{\mathcal{T}})\leq \delta,
\end{equation*}
then
\begin{equation}\label{111905}
d_{p}(w(t),{\mathcal{T}})\leq 2\delta,
\qquad\mbox{for all}\ |t|\leq \delta^ {-\mathcal{M}}.
\end{equation}
\end{thm}


\begin{proof}Take $\delta<\rho$ and $\delta^{1/10}<\tilde\eta<1$. Let
\begin{equation}\label{091302}
\delta^{-\frac{\mathcal{M}+1}{p-1}}\leq\mathcal{N}+1<\delta^{-\frac{\mathcal{M}+1}{p-1}}+1.
\end{equation}Then based on Theorem \ref{thm7.1}, we obtain a partial normal form of order $\mathcal{M}+2$
\begin{equation*}\breve{\breve{H}}(x,y,q,\bar q;\xi):=
\breve H\circ\Phi={\breve N}(y,q,\bar q;\xi)+{Z}(y,q,\bar q;\xi)+{
P}(x,y,q,\bar q;\xi)+{ Q}(x,y,q,\bar q;\xi).
\end{equation*}
To obtain the estimate (\ref{111905}), it is sufficient to prove that
\begin{equation}\label{091308'}
|||X_{{
P}}|||_{\mathcal{P}^p,D(s_0/4,4\delta,4\delta)\times\Pi_{\eta}}\leq
\delta^{\mathcal{M}+\frac12}
 \end{equation}
 and
\begin{equation}\label{091308}
|||X_{{
Q}}|||_{\mathcal{P}^p,D(s_0/4,4\delta,4\delta)\times\Pi_{\eta}}\leq
\delta^{\mathcal{M}+\frac12}.
 \end{equation}
By a direct calculation,
\begin{eqnarray}
&&\nonumber c\delta\left(\frac1{\tilde\eta^2}\mathcal{N}^{6(\mathcal{M}+7)^2}\delta\right)^\mathcal{M}\\
&=&\nonumber\delta^{\mathcal{M}+1}\left(\frac{c\mathcal{N}^{6(\mathcal{M}+7)^2\mathcal{M}}}{\tilde\eta^{2\mathcal{M}}}\right)
\\\nonumber&\leq&\delta^{\mathcal{M}+1}\left(\frac c{\tilde\eta^{2\mathcal{M}}}\delta^{-\frac{6(\mathcal{M}+7)^2\mathcal{M}(\mathcal{M}+1)}{p-1}}\right)
\qquad\mbox{(in view of the inequality (\ref{091302}))}\\
&\leq&\nonumber\delta^{\mathcal{M}+1}\left(\frac
c{\tilde\eta^{2\mathcal{M}}}\delta^{-\frac{1}{4}}\right)\qquad \mbox{(in view of $p\geq24(\mathcal{M}+7)^4+1$)}\\
&\leq&\label{091303} \delta^{\mathcal{M}+\frac58}\qquad \mbox{(by
assuming $\delta$ is very small)}.
\end{eqnarray}
In view of the inequalities (\ref{091304}) and (\ref{091303}), we
have
\begin{equation}\label{091307}
|||X_{{
P}}|||^T_{p,D(s_0/4,4\delta,4\delta)\times\Pi_{\eta}}\leq
\delta^{\mathcal{M}+\frac12}.
 \end{equation}
 Moreover, by the inequalities (\ref{081105}) and
 (\ref{091307}), we finish the proof of (\ref{091308'}).

On the other hand,
\begin{eqnarray}
\nonumber||\hat{z}||_1&=&\sqrt{\sum_{|j|\geq\mathcal{N}+1}|z_j|^2j^2}\\
&=&\nonumber\sqrt{\sum_{|j|\geq\mathcal{N}+1}|z_j|^2j^{2p}/j^{2(p-1)}}\\
&\leq&\nonumber\frac{||\hat{z}||_p}{(\mathcal{N}+1)^{p-1}}\\
&\leq&\label{091306}\delta^{\mathcal{M}+1}||\hat{z}||_p\qquad\qquad\qquad
\mbox{( by (\ref{091302}))}.
\end{eqnarray}
Note that ${ Q}(x,y,q,\bar q;\xi)=O(||\hat q||^3_p)$. Then in view of (\ref{011602})-(\ref{012202}) and (\ref{091306}),
we finish the proof of (\ref{091308}).
\end{proof}

\section{Proof of Theorem \ref{T4}}

\textbf{Step 1.  Write equation (\ref{26}) as an infinite dimensional  Hamiltonian system.}\\

Introducing coordinates $q,\tilde q\in\ell^2_p$ by
\begin{equation*}
u=\sum_{j\geq 1}\frac{q_j}{\sqrt{\lambda_j}}\phi_j,\qquad v:=u_t=\sum_{j\geq1}\tilde q_j\sqrt{\lambda_j}\phi_j,\qquad \lambda_j(\xi)=\sqrt{\mu_j+\xi_j},
\end{equation*}
where $\mu_j=j^2+m$ and $\phi_j(x)=\sqrt{2/\pi}\sin
jx$ are respectively
the simple eigenvalues and eigenvectors of
$-\partial_{xx}+m$ with Dirichlet boundary
conditions. The Hamiltonian of (\ref{26}) is
\begin{eqnarray*}
H_{NLW}&=&\int_0^{\pi}\left(\frac{v^2}2+\frac12\left(u_x^2+(m+M_{\xi})u\right)+\frac14u^4\right)dx\nonumber\\
&=&\frac12\sum_{j\geq1}\lambda_j(q_j^2+\tilde {q}_j^2)+G(q),
\end{eqnarray*}
where
\begin{equation*}
G(q)=\frac{1}4\sum_{i\pm j\pm k\pm l=0}G_{ijkl}q_iq_jq_kq_l
\end{equation*}
with
\begin{equation*}
G_{ijkl}=\frac{1}{\sqrt{\lambda_i\lambda_j\lambda_k\lambda_l}}\int_0^{\pi}\phi_i\phi_j\phi_k\phi_ldx.
\end{equation*}

Introduce complex coordinates
\begin{equation*}
w_j:=\frac1{\sqrt{2}}(q_j+\sqrt{-1}\tilde q_j),\qquad \bar w_j:=\frac1{\sqrt{2}}(q_j-\sqrt{-1}\tilde q_j),
\end{equation*}
and let $\underline{z}=((w_j)_{j\geq1},(\bar w_j)_{j\geq1})$.
Following Example 3.2 in \cite{BG}, it is proven that there exists
a constant $c_p>0$ such that
\begin{equation}\label{060701}
||\widetilde{X_{G}}(\underline{z}^{(1)},\underline{z}^{(2)},\underline{z}^{(3)})||_{p+1}\leq
c_p||(\underline{z}^3)||_{p,1}.
\end{equation}
In particular, when $p=1$, the inequality (\ref{060701}) reads
\begin{equation}\label{091901}
||\widetilde{X_{G}}(\underline{z}^{(1)},\underline{z}^{(2)},\underline{z}^{(3)})||_1\leq ||\widetilde{X_{G}}(\underline{z}^{(1)},\underline{z}^{(2)},\underline{z}^{(3)})||_2\leq
c_1||(\underline{z}^3)||_{1,1}.
\end{equation}
The inequalities (\ref{060701}) and (\ref{091901}) show that the
Hamiltonian vector filed $X_G(\underline z)$ has $p$-tame property.

\textbf{Step 2. Introduce the action-angle variables.}\\
Without loss of generality, we choose
$\phi_{1},\phi_{2},\dots,\phi_{n}$ as tangent direction and
the other as normal direction. Let
\begin{equation}\label{091902}
\tilde w=(w_{1},\dots,w_{n}) \qquad\mbox{and}\qquad \hat w
=(w_j)_{j\geq n+1}
\end{equation}be the tangent variable and normal variable, respectively.  Then rewrite $G(w,\bar w)$ in the multiple-index as
\begin{eqnarray}\label{060103}
G(w,\bar{w})=\sum_{|\mu|+|\nu|+|\beta|=4}G^{\mu\nu\beta}\tilde
w^{\mu}\bar{\tilde w}^{\nu}z^{\beta},\qquad
\mu,\nu\in\mathbb{N}^n,\beta\in\mathbb{N}^{\mathbb{N}},
\end{eqnarray} where $z=(\hat w,\bar{\hat w})$ and $G^{\mu\nu\beta}=G_{ijkl}$ for some corresponding $i,j,k,l.$

Introduce action-coordinates on the first $n$ modes
\begin{equation*}
w_j:=\sqrt{I_j+y_j}e^{\sqrt{-1}x_j},\qquad 1\leq j\leq n
\end{equation*}
with
\begin{equation*}
I_j\in\left(\frac{r^{2\theta}}2,r^{2\theta}\right],\qquad \theta\in(0,1).
\end{equation*}

Then the symplectic
structure is $dy\wedge dx+\sqrt{-1}d\hat w\wedge d\bar{\hat w}$, where
$y=(y_1,\dots,y_n)$ and $x=(x_1,\dots,x_n)$.
 Hence, (\ref{060103})
is changed into
\begin{eqnarray}
G(w,\bar w)&=&\nonumber G(x,y,
z)\\&=&\nonumber\sum_{|\mu|+|\nu|+|\beta|=4}2^{\frac12(|\mu|+|\nu|)}G^{\mu\nu\beta}\sqrt{(\zeta+y)^{\mu}}e^{\sqrt{-1}\langle\mu,x\rangle}
\sqrt{(\zeta+y)^{\nu}}e^{-\sqrt{-1}\langle\nu,x\rangle}z^{\beta}\\
\nonumber
&=&\sum_{|\mu|+|\nu|+|\beta|=4}(2^{\frac12(|\mu|+|\nu|)}G^{\mu\nu\beta}
e^{\sqrt{-1}\langle\mu-\nu,x\rangle})\sqrt{(\zeta+y)^{\mu+\nu}}
z^{\beta}\nonumber\\
&=&\sum_{|\beta|\leq4}G^{\beta}(x,y) z^{\beta}\label{121001},
\end{eqnarray}
where \begin{eqnarray*}\zeta=(\zeta_1,\dots,\zeta_n),\quad G^{\beta}(x,y)=\sum_{|\mu|+|\nu|=4-|\beta|}
2^{\frac12(|\mu|+|\nu|)}G^{\mu\nu\beta}e^{\sqrt{-1}\langle\mu-\nu,x\rangle}\sqrt{(\zeta+y)^{\mu+\nu}}.
\end{eqnarray*}

\textbf{Step 3. Show that the Hamiltonian vector field $X_G$ has $p$-tame property.}\\
Note $G$ has $p$-tame property in the coordinate $\underline{z}$, and introducing action-angle variables is a coordinate transformation of finite variables, so $G$ has $p$-tame property in the coordinate $(x,y,z)$. More precisely,
 assume
$G(x,y,z)=\sum_{|\beta|\leq4}G^{\beta}(x,y) z^{\beta}$ is defined on
the domain $(x,y,z)\in D(s,r,r)$ for some
$0<s,r\leq 1$. In this step, we will show that the Hamiltonian vector field
$X_G(x,y,z)$ has $p$-tame property on the domain $D(s,r,r)$.

Firstly, we would like to show
\begin{equation*}
|||G_{ z}|||^{T}_{p,D(s,r,r)\times\Pi}<\infty.
\end{equation*}
In view of (\ref{060103}),
\begin{eqnarray*}
G_{z}(x,y,z)&=&G_z(w,\bar{w})\nonumber\\
&=&\sum_{|\mu|+|\nu|+|\beta|=4}\beta G^{\mu\nu\beta}\tilde
w^{\mu}\bar{\tilde w}^{\nu}z^{\beta-1}.
\end{eqnarray*}Let
\begin{equation*}
G_{z}^{\beta-1}(x,y)=G_{z}^{\beta-1}(\tilde w,\bar{\tilde
w})=\sum_{|\mu|+|\nu|=4-|\beta|}\beta G^{\mu\nu\beta}\tilde
w^{\mu}\bar{\tilde w}^{\nu},
\end{equation*}
and then
\begin{equation*}
G_{z}(x,y,z)=\sum_{|\beta|\leq 4}G_{z}^{\beta-1}(x,y)z^{\beta-1}.
\end{equation*}
Hence, we obtain
\begin{eqnarray}\label{091910}
||\widetilde{\lfloor G_z\rceil}_{D(s,r)\times\Pi}||_{p+1}\leq
c(\sum_{1\leq i\leq 3}||{z}^{(i)}||_p+\sum_{1\leq i\neq j\leq
3}||{z}^{(i)}||_1 ||{z}^{(j)}||_p+||({z}^{3})||_{p,1}),
\end{eqnarray}
where $c>0$ is a constant depending on $s,r,n$ and $p$, and the
above inequality is based on the inequality (\ref{060701}) and the
definition of ${\lfloor \cdot\rceil}_{D(s,r)\times\Pi}$ (see
(\ref{092703}) for the details), and noting that
$\underline{z}=(\tilde w,\hat w,\bar{\tilde w},\bar {\hat w})$ and $z=(\hat w,\bar
{\hat w})$. In particular, when $p=1$, the inequality (\ref{091910}) reads
\begin{eqnarray}\label{091911}
||\widetilde{\lfloor G_z\rceil}_{D(s,r)\times\Pi}||_1\leq \breve
c(\sum_{1\leq i\leq 3}||{z}^{(i)}||_1+\sum_{1\leq i\neq j\leq
3}||{z}^{(i)}||_1 ||{z}^{(j)}||_1+||({z}^{3})||_{1,1}),
\end{eqnarray}where $\breve c>0$ is a constant depending on
$s,r$ and $n$. Based on the inequalities (\ref{091910}) and
(\ref{091911}), we obtain
\begin{equation}\label{060203}
|||G_{z}|||^{T}_{p,D(s,r,r)\times \Pi}<\infty.
\end{equation}
Similarly, we obtain
 \begin{equation}\label{091909}
 |||G_{x}|||_{D(s,r,r)\times\Pi}<\infty
 \end{equation}
 and
\begin{eqnarray}\label{060202} |||G_y|||_{D(s,r,r)\times\Pi}<\infty.
\end{eqnarray}
In view of the inequalities (\ref{060203}), (\ref{091909}) and (\ref{060202}), we get
\begin{equation}\label{091912}
|||X_G|||^{T}_{p,D(s,r,r)\times\Pi}<\infty.
\end{equation}
Finally, we obtain a Hamiltonian $H(x,y,q,\bar q;\xi)$ having the
following form
\begin{equation}H(x,y,q,\bar q;\xi)=N(x,y,q,\bar q;\xi)+R(x,y,q,\bar
q;\xi),\end{equation} where
\begin{equation}
N(x,y,q,\bar q;\xi)=H_{0}(w,\bar w)=\sum_{j=1}^n
\omega_j(\xi)y_j+\sum_{j\geq 1}\Omega_j(\xi)q_j\bar q_j,
\end{equation}with the tangent frequency \begin{equation}\label{001}
\omega({\xi})=(\omega_j(\xi))_{1\leq j\leq n}, \qquad \omega_j=\sqrt{j^2+m+\xi_{j}},
\end{equation} the normal frequency
\begin{equation}\label{002}
\Omega(\xi)=(\Omega_1(\xi),\Omega_2(\xi),\dots,),\qquad
\Omega_j(\xi)=\sqrt{(j+n)^2+m+\xi_{j+n}},\quad j\geq1
\end{equation}
and the perturbation
\begin{equation}\label{003}
R(x,y,q,\bar q;\xi)=\varepsilon G(x,y,q,\bar q;\xi).
\end{equation}

In view of  (\ref{091912}), (\ref{001})-(\ref{003}), all assumptions in Theorem \ref{T1} hold. According to Theorem \ref{T1}, we obtain a KAM normal form of order 2 where the nonlinear terms satisfy $p$-tame property.

Furthermore, let $\delta$ be given in the statement of Theorem \ref{thm7.1} and $\mathcal{N}$ be given in (\ref{091302}). Take $\eta$ satisfying $\delta^{1/6}<
\tilde\eta<{\mathcal{N}}^{-3}$. Then we obtain a KAM partial normal form of order $\mathcal{M}+2$ where the nonlinear terms satisfy $p$-tame property.

Finally, based on Corolarry \ref{T3}, for each $\xi\in\Pi_{\eta}\subset\Pi_{\infty}$, the KAM
torus $\mathcal{T}_{\xi}$ for equation (\ref{26}) is stable in a long time, i.e.
for any solution
$u(t,x)$ of equation (\ref{26}) with the
initial datum satisfying \color{black}
$${d}_{H^p_0[0,\pi]}(u(0,x),\mathcal{T}_{\xi})\leq \delta,$$
then
\begin{equation*}{d}_{H^p_0[0,\pi]}(u(t,x),\mathcal{T}_{\xi})\leq
2\delta,\qquad \mbox{for all} \ |t|\leq
{\delta}^{-\mathcal{M}}.
\end{equation*}

\section*{Acknowledgements}{\it In the Autumn of 2007,
Professor H. Eliasson gave a series of lectures on KAM theory for
Hamiltonian PDEs in Fudan University. In those lectures, he proposed
to study the normal form in the neighborhood of the invariant tori
and the nonlinear stability of the invariant tori.  The
authors are heartily grateful to Professor Eliasson.

The authors are also heartily grateful to Professor Bambusi and Professor Yuan for valuable discussions and suggestions. }

\section{Appendix: Properties of the Hamiltonian with $p$-tame property}

The following lemma shows that $p$-tame property persists under
Poisson brackets, which can be parallel proved following the proof
of Theorem 3.1 in \cite{CLY}:

\begin{lem}\label{021102}
Suppose that both Hamiltonian functions
$$U(x,y,z;\xi)=\sum_{\beta\in\mathbb{N}^{\bar{\mathbb{Z}}}}U^{\beta}(x,y;\xi)z^{\beta}$$
and
$$V(x,y,z;\xi)=\sum_{\beta\in\mathbb{N}^{\bar{\mathbb{Z}}}}V^{\beta}(x,y;\xi)z^{\beta},$$
satisfy $p$-tame property on the domain $D(s,r,r)\times\Pi$, where
\begin{equation*}
U^{\beta}(x,y;\xi)=\sum_{\alpha\in\mathbb{N}^{{n}}}U^{\alpha\beta}(x;\xi)y^{\alpha}
\end{equation*} and
\begin{equation*}
V^{\beta}(x,y;\xi)=\sum_{\alpha\in\mathbb{N}^{{n}}}V^{\alpha\beta}(x;\xi)y^{\alpha}
.\end{equation*} Then the Poisson bracket $\{U,V\}(x,y,z;\xi)$ of
$U(x,y,z;\xi)$ and $V(x,y,z;\xi)$ with respect to the symplectic
structure $\sum_{1\leq j\leq n}dy_j\wedge dx_j+\sqrt{-1}\sum_{j\geq
1}dz_{-j}\wedge dz_j$ has $p$-tame property on the domain
$D(s-\sigma,r-\sigma',r-\sigma')\times\Pi$ for
$0<\sigma<s,0<\sigma'<r/2$. Moreover, the following inequality holds
\begin{eqnarray}&&|||X_{\{U,V\}}|||_{p,D(s-\sigma,r-\sigma',r-\sigma')\times\Pi}^T
\nonumber\\
\label{091120}&\leq&C\max\left\{\frac 1{\sigma},\frac
{r}{\sigma'}\right\}|||X_{U}|||_{p,D(s,r,r)\times\Pi}^T
|||X_{V}|||_{p,D(s,r,r)\times\Pi}^T,
\end{eqnarray}
where $C>0$ is a constant depending on $n$.
\end{lem}

Denote $X_U^t $ by the flow of the Hamiltonian vector field of
$U(x,y,z;\xi)$. It follows from Taylor¡¯s formula that
\begin{equation}\label{081108}
V\circ X_U^t(x,y,z;\xi)=\sum_{j\geq 0}
\frac{t^{j}}{j!}V^{(j)}(x,y,z;\xi),
\end{equation}
where
\begin{equation*}
V^{(0)}(x,y,z;\xi):=V(x,y,z;\xi),\qquad
V^{(j)}(x,y,z;\xi):=\{V^{(j-1)},U\}(x,y,z;\xi).
\end{equation*}
Then based on (\ref{091120}) in Lemma {\ref{021102}} and
(\ref{081108}), we have the following estimate of symplectic
transformation, which can be parallel proved following the proof of
Theorem 3.3 in \cite{CLY}:

\begin{lem}\label{081914}
Consider two Hamiltonians $U(x,y,z;\xi)$ and $V(x,y,z;\xi)$
satisfying $p$-tame property on the domain $D(s,r,r)\times\Pi$ for
some $0<s,r\leq1$. Given $0<\sigma<s, 0<\sigma'<r/2$, suppose
\begin{equation*}\label{081112}|||X_U|||_{p,D(s,r,r)\times\Pi}^T\leq \frac1{2A},
\end{equation*}
where
\begin{equation*}\label{090504}A=4Ce\max\left\{\frac{1}{\sigma},\frac{r}{\sigma'}\right\}
\end{equation*}
and $C>0$ is the constant given in (\ref{091120}) in Theorem
\ref{021102}. Then for each $|t|\leq 1$, we have
\begin{equation*}
|||X_{V\circ
X_U^t}|||_{p,D(s-\sigma,r-\sigma',r-\sigma')\times\Pi}^T\leq
2|||X_V|||_{p,D(s,r,r)\times\Pi}^T.
\end{equation*}
\end{lem}

The following lemma estimates $p$-tame norm of the solution of
homological equation during KAM iterative procedure and normal form
iterative procedure, which can be parallel proved following the
proof of Theorem 3.4 in \cite{CLY}:

\begin{lem}\label{0005}Consider two Hamiltonians
$$U(x,y,z;\xi)=\sum_{\alpha\in\mathbb{N}^n,\beta\in\mathbb{N}^{\bar{\mathbb{Z}}}}
U^{\alpha\beta}(x;\xi)y^{\alpha}z^{\beta}$$ and
$$V(x,y,z;\xi)=\sum_{\alpha\in\mathbb{N}^n,\beta\in\mathbb{N}^{\bar{\mathbb{Z}}}}
V^{\alpha\beta}(x;\xi)y^{\alpha}z^{\beta}.$$ Suppose $V(x,y,z;\xi)$
has $p$-tame property on the domain $D(s,r,r)\times \Pi$, i.e
$$|||X_V|||_{p,D(s,r,r)\times\Pi}^T<\infty.$$ For each
$\alpha\in\mathbb{N}^n,\beta\in\mathbb{N}^{\bar{\mathbb{Z}}},k\in
\mathbb{Z}^n, j\geq 1$ and some fixed constant $\tau>0$, assume the
following inequality holds
\begin{equation*}\label{0007}
|\widehat{U^{\alpha\beta}}(k;\xi)|+|\partial_{\xi_j}\widehat{U^{\alpha\beta}}(k;\xi)|\leq
(|k|+1)^{\tau}(|\widehat{V^{\alpha\beta}}(k;\xi)|+|\partial_{\xi_j}\widehat{V^{\alpha\beta}}(k;\xi)|)
,\end{equation*} where $\widehat{U^{\alpha\beta}}(k;\xi)$ and
$\widehat{V^{\alpha\beta}}(k;\xi)$ are the $k$-th Fourier
coefficients of $U^{\alpha\beta}(x;\xi)$ and
$V^{\alpha\beta}(x;\xi)$, respectively. Then, $U(x,y,z;\xi)$ has
$p$-tame property on the domain $D(s-\sigma,r,r)\times\Pi$ for
$0<\sigma<s$. Moreover, we have
\begin{equation}\label{0012} |||X_U|||_{p,D(s-\sigma,r,r)\times\Pi}^T\leq
\frac{c}{\sigma^{\tau}} |||X_V|||_{p,D(s,r,r)\times\Pi}^T,
\end{equation}
where $c>0$ is a constant depending on $s$ and $\tau$.
\end{lem}

The following lemma compares $p$-tame norm with the usual weighted
norm for Hamiltonian vector field, which can be parallel proved
following the proof of Theorem 3.5 in \cite{CLY}:

\begin{lem}\label{012002}Given a Hamiltonian
\begin{equation*}\label{100802}U(x,y,z;\xi)=\sum_{\beta\in\mathbb{N}^{\bar{\mathbb{Z}}}}
U^{\beta}(x,y;\xi)z^{\beta} \end{equation*} satisfying $p$-tame
property on the domain $D(s,r,r)\times\Pi$ for some $0<s,r\leq 1$.
Then we have
\begin{equation}\label{081105}
|||X_U|||_{\mathcal{P}^p,D(s,r,r)\times\Pi}\leq
|||X_U|||_{p,D(s,r,r)\times\Pi}^T.
\end{equation}
\end{lem}

Then based on (\ref{081105}) in Lemma {\ref{012002}} and the proof
of Lemma A.4 in \cite{P1}, we have the following lemma:

\begin{lem}\label{081503}
Suppose the Hamiltonian
$$U(x,y,z;\xi)=\sum_{\beta\in\mathbb{N}^{\bar{\mathbb{Z}}}}U^{\beta}(x,y;\xi)z^{\beta}$$
has $p$-tame property on the domain $D(s,r,r)\times\Pi$ for some
$0<s,r\leq 1$. Let $X_U^{t}$ be the phase flow generalized by the
Hamiltonian vector field $X_U$. Given $0<\sigma<s$ and
$0<\sigma'<r/2$, assume
\begin{equation*}\label{091810}|||X_U|||_{p,D(s,r,r)\times\Pi}^T<\min\{\sigma,\sigma'\}.
\end{equation*}
Then, for each $\xi\in\Pi$ and each $|t|\leq 1$, one has
\begin{equation}\label{091101}
||X_U^t-id||_{p,D(s-\sigma,r-\sigma',r-\sigma')}\leq
|||X_U|||_{p,D(s,r,r)\times\Pi}^T,
\end{equation}
where
\begin{equation}\label{081102}
||X_U^t-id||_{p,D(s-\sigma,r-\sigma',r-\sigma')}=\sup_{w\in
D(s-\sigma,r-\sigma',r-\sigma')}||(X_U^t-id)w||_{\mathcal{P}^{p},D(s,r,r)}.
\end{equation}
\end{lem}

\end{document}